\documentclass[final]{siamltex}
\usepackage{amssymb}
\usepackage{amsmath}
\usepackage{amscd}
\usepackage{latexsym}
\usepackage{cite}
\usepackage{epic}
\usepackage{eepic}
\usepackage{fancybox}
\usepackage[dvips]{graphicx} 
\usepackage{color}
\usepackage{psfrag}
\usepackage[draft]{showkeys}
\usepackage{ifthen}
\usepackage[dvips,bookmarks=true]{hyperref}

\providecommand{\Bdsymb}{\Gamma}                   
\providecommand{\Div}{\operatorname{div}}          
\providecommand{\bDiv}{{\Div}_{\Bdsymb}}           
\providecommand{\curl}{\operatorname{{\bf curl}}}  
\providecommand{\bcurl}{{\curl}_{\Bdsymb}}         
\providecommand{\grad}{\operatorname{{\bf grad}}}       
\providecommand{\bgrad}{\grad_{\Bdsymb}}                

\newcommand{\Vn}{{\mathbf{n}}}

\newcommand{\Vu}{{\mathbf{u}}}
\newcommand{\Vv}{{\mathbf{v}}}

\newcommand{\Vx}{{\mathbf{x}}}

\providecommand{\Ba}{{\boldsymbol{a}}}

\providecommand{\Bn}{{\boldsymbol{n}}}

\providecommand{\Bt}{{\boldsymbol{t}}}

\providecommand{\Bx}{{\boldsymbol{x}}}

\newcommand{\VA}{{\mathbf{A}}}

\newcommand{\VH}{{\mathbf{H}}}

\newcommand{\VV}{{\mathbf{V}}}
\newcommand{\VW}{{\mathbf{W}}}
\newcommand{\VX}{{\mathbf{X}}}

\newcommand{\Pibf}{\boldsymbol{\Pi}}

\newcommand{\Phibf}{\boldsymbol{\Phi}}

\providecommand{\Cd}{{\cal D}}
\providecommand{\Ce}{{\cal E}}

\providecommand{\Cm}{{\cal M}}

\providecommand{\Cp}{{\cal P}}

\providecommand{\Cs}{{\cal S}}

\newcommand{\Crv}{\boldsymbol {\cal R}}

\newcommand{\Ctv}{\boldsymbol {\cal T}}

\providecommand{\bbN}{\mathbb{N}}

\providecommand{\bbR}{\mathbb{R}}

\newcommand*{\DP}[2]{\left<{#1},{#2}\right>} 

\providecommand*{\wh}[1]{\widehat{#1}}

\providecommand*{\N}[1]{\left\|{#1}\right\|} 
\newcommand*{\SN}[1]{\left|{#1}\right|}      

\newcommand*{\Hcurl}[1][\defaultdomain]{\boldsymbol{H}(\curl,{#1})}
\newcommand*{\bHcurl}[2][\defaultdomain]{\boldsymbol{H}_{#2}(\curl,{#1})}
\newcommand*{\zbHcurl}[1][\defaultdomain]{\bHcurl[#1]{0}}

\title{Convergence of the Natural $hp$-BEM for the 
  Electric Field Integral Equation on Polyhedral Surfaces
  \thanks{A.B. and N.H. acknowledge support by EPSRC under grant no. EP/E058094/1.}}
\author{A. Bespalov\thanks{Department of Mathematical Sciences, Brunel University,
                           Uxbridge, West London UB8 3PH, UK,
                           albespalov\symbol{64}yahoo.com}
  \and 
  N. Heuer\thanks{Facultad de Matem\'aticas, Pontificia Universidad Cat\'olica de Chile,
                  Avenida Vicu\~na Mackenna 4860, Santiago, Chile,
                  nheuer\symbol{64}mat.puc.cl} 
  \and 
  R. Hiptmair\thanks{SAM, ETH Z\"urich, CH-8092 Z\"urich, Switzerland,
                     hiptmair\symbol{64}sam.math.ethz.ch}}

\newtheorem{assumption}[theorem]{Assumption}
\newtheorem{remark}[theorem]{Remark}

\begin{document}
\DeclareGraphicsRule{.eps.bz2}{eps}{.eps.bb}{`bunzip2 -c #1}
\DeclareGraphicsRule{.ps.bz2}{eps}{.ps.bb}{`bunzip2 -c #1}
\DeclareGraphicsRule{.eps.bz}{eps}{.eps.bb}{`bunzip2 -c #1}
\DeclareGraphicsRule{.eps.gz}{eps}{.eps.bb}{`gunzip -c #1}
\DeclareGraphicsRule{.ps.bz}{eps}{.ps.bb}{`bunzip2 -c #1}
\DeclareGraphicsRule{.ps.gz}{eps}{.ps.bb}{`gunzip -c #1}

\maketitle

\begin{abstract}
  We consider the variational formulation of the electric field integral equation
  (EFIE) on bounded polyhedral open or closed surfaces. We employ a conforming
  Galerkin discretization based on $\bDiv$-conforming Raviart-Thomas boundary
  elements (BEM) of locally variable polynomial degree on shape-regular surface
  meshes. We establish asymptotic quasi-optimality of Galerkin solutions on
  sufficiently fine meshes or for sufficiently high polynomial degree.
\end{abstract}



\pagestyle{myheadings}
\thispagestyle{plain}
\markboth{A. Bespalov, N. Heuer, R. Hiptmair}{$hp$-BEM for EFIE}

\section{Introduction}
\label{sec:introduction}

Let $\Gamma$ be a piecewise flat (open or closed) orientable surface equipped with a
conforming triangulation $\Cm=\{K\}$, consisting of triangles. Throughout, uniform
bounds on the shape-regularity of the cells will be tacitly taken for granted, see
\cite[Ch.~3, \S~3.1]{CIA78}. For a fixed wave number $k>0$, let $V_{k}$ and $\VV_{k}$
stand for the scalar or vectorial single layer boundary integral operator on $\Gamma$
for the Helmholtz operator $-\Delta - k^{2}$, see \cite[Sect.~4.1]{BCS01} or
\cite[Sect.~5]{BUH03}.  The bilinear form underlying the variational formulation of
the electric field integral equation (``Rumsey's principle'') reads
(see \cite{BEN84,HAR89}, \cite[Sect.~4.2]{BCS01} or \cite[Sect.~7.2]{BUH03} for closed
surfaces, ``boundaries'', and \cite[Sect.~3]{BUC01} for open surfaces, ``screens'')
\begin{gather}
  \label{eq:a}
  a(\Vu,\Vv) := \DP{V_{k}\bDiv \Vu}{\bDiv \Vv}_{\Gamma} -
  k^{2}\DP{\VV_{k}\Vu}{\Vv}_{\Gamma}\;,
\end{gather}
where, as discussed in \cite{BCS00}, $\DP{\cdot}{\cdot}_{\Gamma}$ hints at a duality
pairing, extending the $L^{2}(\Gamma)$-pairing for tangential vector fields or
functions on $\Gamma$. The variational problem is posed on the Hilbert space
\begin{gather}
  \label{eq:X1}
  \VX = \VH_{\parallel}^{-1/2}(\bDiv,\Gamma)\;,
\end{gather}
in the case of a boundary $\Gamma=\partial\Omega$, $\Omega\subset\bbR^{3}$ a Lipschitz
polyhedron, or on 
\begin{gather}
  \label{eq:X2}
  \VX = \{\Vu\in \VH_{\parallel}^{-1/2}(\bDiv,\Gamma):\,
  \DP{\Vu}{\bgrad v} + \DP{\bDiv \Vu}{v} = 0 \;\forall v\in C^{\infty}(\overline{\Gamma})\}
\end{gather}
in the case of a screen $\Gamma$. The latter space can be understood as a space of
$\bDiv$-conforming tangential surface vector fields with vanishing in-plane normal
component on the screen edge $\partial\Gamma$.  We refer to
\cite{BUC99,BUC99a,BCS00,BUF02p,BUH03} for a definition and more information about
$\VH_{\parallel}^{-1/2}(\bDiv,\Gamma)$ and other trace spaces.  An in-depth discussion for
screens is given in \cite[Sect.~2]{BUC01}.  In this article we adopt the notations of
\cite{BUC99}. Further, the two situations of open and closed surfaces will be treated
in parallel.

The EFIE can be recast as a linear variational problem for $a(\cdot,\cdot)$ on $\VX$:
given a source functional $f\in \VX'$ it reads
\begin{gather}
  \label{eq:lvp}
  \Vu\in \VX:\quad a(\Vu,\Vv) = f(\Vv)\quad\forall \Vv\in \VX\;.
\end{gather}
In order to ensure uniqueness of the solution to \eqref{eq:lvp},
we make the following assumption \cite[Sect.~7.1]{BUH03}.

\begin{assumption}
  In the case of a closed surface $\Gamma=\partial\Omega$ we assume that $k$ is
  different from a Dirichlet eigenvalue of the operator $\curl\curl$ on $\Omega$.
\end{assumption}

We opt for a natural boundary element (BE) Galerkin discretization based on
conforming trial and test spaces $\VX_{N}\subset \VX$. These are obtained by using
$\VH(\Div_\Gamma,\Gamma)$-conforming Raviart-Thomas spaces of variable local polynomial
degrees $p_{K}\in\bbN_{0}$, $K\in\Cm$, on the surface triangulation $\Cm$,
see Sect.~\ref{sec:bound-elem-spac} for a precise definition.

Refinement of the BE spaces can be achieved by raising the local polynomial degrees
$p_{K}$ ($p$-refinement) or reducing the sizes $h_{K}$ of the cells of $\Cm$
($h$-refinement). Thus, the proposed discretization qualifies as ``$hp$-boundary
element method (BEM)''.

Roughly speaking, judicious $hp$-refinement can be expected to offer exponential
convergence of Galerkin solutions even when the exact solution lacks global smoothness
\cite{SAB98}. $hp$-BEM approaches have been suggested for various boundary integral
equations \cite{BEH08a,DEO05,BES08,HEMS99,HES96b,GUH06} and are a natural idea for
the EFIE as well. While convergence theory for $h$-refinement is well established
\cite{BEN84,BEN84a,HIS01,BUC01}, the extension to $hp$-refinement proved to be
difficult, see \cite{BEH09a,BEH09} for partial results.

This article fills the gap and proves the following convergence result that
translates into an a priori error estimate in the natural ``energy norm'' provided
that information about some smoothness of the solution $\Vu$ of \eqref{eq:lvp} is
available, \textit{cf.} \cite[Sect.~8]{HIS01} and \cite[Sect.~4]{BUC01} for the
$h$-version, \cite{BEH09} for the $p$-version, and \cite{BEH09a} for the $hp$-version
with quasi-uniform meshes.

\begin{theorem}
  \label{thm:main}
  There is a constant $C_{0}>0$ such that for any $f \in \VX'$ and for arbitrary
  mesh-degree combination satisfying 
  $\max\limits_{K} \sqrt{\frac{h_{K}}{p_{K}+1}} < C_{0}$
  the Galerkin BE discretization of \eqref{eq:lvp} admits a unique solution
  $\Vu_{N} \in \VX_N$ and the Galerkin $hp$-BEM converges quasi-optimally, i.e.,
  \begin{gather}
    \label{eq:1}
    \N{\Vu-\Vu_{N}}_{\VX} \leq C \inf\limits_{\Vv_{N}\in \VX_{N}} \N{\Vu-\Vv_{N}}_{\VX}\;.
  \end{gather}
  Both constants $C_{0}$ and $C$ may depend only on the geometry of $\Gamma$ and the
  shape-regularity of the surface triangulation $\Cm$.
\end{theorem}

We remark that the policy of the proof has a lot in common with recent proofs of
discrete compactness for the $p$-version of edge elements \cite{HIP08t,BCD06,BDC03}.
The main tools are the same, namely, the sophisticated mathematical inventions of
regularizing lifting operators \cite{CMI08} (see Sect.~\ref{sec:smooth-poinc-mapp} below)
and projection based interpolation operators \cite{DEB04,DEB01} (see
Sect.~\ref{sec:proj-based-interp}). They pave the way for verifying the assumptions
of an abstract theory of Galerkin approximations for non-coercive variational problems,
see \cite{BUF03} and Sects.~\ref{sec:abstract-theory}, \ref{sec:discrete-splitting} below.

Building on these mighty foundations the present article cannot be and does not aspire to be
self-contained, but will give detailed references to relevant literature.  We refer
to \cite{BEH08} for an earlier version of this paper whose analysis is based on a
Hodge-decomposition of $\VX$ which, due to regularity issues on non-smooth surfaces,
requires a sophisticated projection based interpolation operator which is not needed
in this paper.

In the sequel, generic constants, designated by $C$, $C_{0}$, $C_{1}$, etc., may
depend only on the geometry of $\Gamma$ and the shape-regularity of $\Cm$. They must not
depend on cell sizes, local polynomial degrees, and any function.


\newcommand{\PR}{\operatorname{\mathsf{R}}}
\newcommand{\LI}{\operatorname{\mathsf{L}}}
\newcommand{\RT}[2]{\Crv\Ctv_{#1}(#2)}
\newcommand{\RTn}[2]{\Crv\Ctv_{#1,0}(#2)}

\section{Boundary element spaces}
\label{sec:bound-elem-spac}

Raviart-Thomas surface elements provide an affine equivalent family of
$\bDiv$-conforming finite elements under the Piola transformation, see
\cite[Sect.~III.3]{BRF91} and \cite{RAT77}. We write $\RT{p}{K}$ for the local
Raviart-Thomas space of order $p$ on the triangle $K\in\Cm$, and $\RTn{p}{K}\subset
\mathbf{H}_{0}(\Div,K)$ for the subspace of local Raviart-Thomas vector fields with
vanishing normal components on $\partial K$.  Vector fields in the latter spaces will
be identified with their extensions by zero onto the whole surface $\Gamma$.

Given a polynomial degree distribution $\{p_{K}:\,p_{K}\in\bbN_{0},\,K\in\Cm\}$, we
define edge degrees according to the ``maximum rule''
\begin{gather}
  \label{eq:ed}
  p_{E} := \max \{ p_{K}:\; K\in\Cm,\; E\subset\overline{K}\}\;,\quad
  E\in\Ce\;,
\end{gather}
where $\Ce$ is the set of edges of $\Cm$. As elaborated in \cite[Sect.~3.4]{HIP02},
Raviart-Thomas spaces can be split into local ``edge contributions'' and ``cell
contributions''. In detail, write $\psi_{1}$ and $\psi_{2}$ for the 
piecewise linear, continuous ``tent/hat functions'' associated with the endpoints
of some edge $E\subset\Ce$. We introduce the edge space $\RT{p_{E}}{E}$ as
\begin{gather}
  \label{eq:18}
  \RT{p_{E}}{E} := \operatorname{span}\{\bcurl (\psi_{1}^{\alpha}\psi_{2}^{\beta}),\;
  \alpha,\beta\in\bbN,\;\alpha+\beta=p_{E}+1\}\;.
\end{gather}
These spaces $\RT{p_{E}}{E}$ obviously satisfy
\begin{gather}
  \label{eq:12}
  \Vu_{N}\in \RT{p_{E}}{E}\quad\Rightarrow\quad
  \operatorname{supp}\Vu_{N}\subset
  \bigcup\{\overline{K}:\,E\subset\overline{K}\}\quad
  \text{and}\quad 
  \boxed{\bDiv\Vu_{N} = 0} \;.
\end{gather}
Then, we \emph{define} the boundary element space (oblivious of boundary conditions!) 
according to 
\begin{gather}
  \label{eq:5}
  \widetilde{\VX}_{N} = \RT{0}{\Cm} + \sum\limits_{E\in\Ce} \RT{p_{E}}{E} + 
  \sum\limits_{K\in\Cm} \RTn{p_{K}}{K}\;.
\end{gather}
Here, the space $\RT{0}{\Cm}$ is the lowest order Raviart-Thomas BE space.  Thanks to
the maximum rule \eqref{eq:ed}, the localized spaces $\VX_{N}(K) :=
{\widetilde{\VX}_{N}|}_{K}$, $K\in\Cm$, fulfil
\begin{gather}
  \label{eq:10}
  \RT{p_{K}}{K} \subset {\VX}_{N}(K)\quad\forall K\in\Cm\quad\text{and}\quad
  {\widetilde{\VX}_{N}\cdot\Vn_{E}|}_{E} = \Cp_{p_{E}}(E)\quad\forall E\in\Ce\;,
\end{gather}
with $\Vn_{E}$ standing for an edge normal, and $\Cp_{p}$ for the space
of (multivariate) polynomials of degree $\leq p$, $p\in\bbN_{0}$. Now, we
are in a position to introduce the $hp$-BEM Galerkin trial and test
spaces: 
\begin{itemize}
\item we pick $\VX_{N}:= \widetilde{\VX}_{N}$ for closed surfaces $\Gamma$,
\item we choose $\VX_{N} := \widetilde{\VX}_{N} \cap \VX$ for screens, that is, in
  order to obtain $\VX_{N}$ edge spaces for edges contained in $\partial\Gamma$ are
  simply discarded as well as basis functions of $\RT{0}{\Cm}$ associated with edges
  on $\partial\Gamma$.
\end{itemize}

Note that for $E\subset\overline{K}$, $K\in\Cm$, we may encounter $p_{E}> p_{K}$,
and, consequently,
\begin{gather}
  \label{eq:16}
  \VX_{N}(K) \not\subset \RT{p_{K}}{K} \subset \VX_{N}(K)\; !
\end{gather}
However, thanks to \eqref{eq:5} and \eqref{eq:12}, we can take for granted
\begin{gather}
  \label{eq:17}
  \bDiv\VX_{N}(K) = \bDiv \RT{p_{K}}{K}\;.
\end{gather}
From \cite[\S III.3, Prop.~3.2]{BRF91} we know that $\bDiv\RT{p}{K} = \Cp_{p}(K)$. Thus, by
\eqref{eq:5}, $\bDiv:\widetilde{\VX}_{N}\mapsto Q_{N}$, where $Q_{N}\subset L^{2}(\Gamma)$
is the space of $\Cm$-piecewise polynomials with degree $p_{K}$ on every $K\in\Cm$.

By \cite[Theorem~3.7]{HIP02}, \cite[Sect.~5.5]{AFW06}, the Raviart-Thomas BE space
$\widetilde{\VX}_{N}$ allows for a \emph{discrete scalar potential space} $\Cs_{N}
\subset C^{0}(\Gamma)$ comprising continuous piecewise polynomial functions on $\Cm$
such that ${\Cs_{N}|}_{E} \subset \Cp_{p_{E}+1}(E)$ for all $E\in\Ce$ and the
localized spaces $\Cs_{N}(K) = {\Cs_{N}|}_{K}$, $K\in\Cm$, satisfy
\begin{gather}
  \label{eq:14}
  \bcurl\Cs_{N}(K) = \VX_{N}(K) \cap \mathbf{H}(\bDiv0,K)\quad\forall K\in\Cm\;.
\end{gather}

Below, we make repeated use of transformation to the reference triangle (``unit
triangle'') $\widehat{K} :=
\operatorname{convex}\{\binom{0}{0},\binom{1}{0},\binom{0}{1}\}$, which is mapped to
a generic $K\in\Cm$ by the affine mapping $\Phibf_{K}:\widehat{K}\mapsto K$,
$\Phibf(\widehat{\Bx}) := \VA_{K}\widehat{\Bx}+\Bt_{K}$,
$\VA_{K}\in\mathbb{R}^{3,2}$, $\Bt_{K}\in\mathbb{R}^{3}$. Writing $\Phibf_{K}^{\ast}$
for the associated co-variant pullback of tangential vector fields, we define spaces
of functions $\widehat{K}\mapsto\mathbb{R}^{2}$,
\begin{gather}
  \label{eq:22}
  \VX_{N}(\widehat{K}) := \Phibf_{K}^{\ast}\VX_{N}(K)\;,
\end{gather}
which, due to non-uniform polynomial degrees, may be different for different cells
$K$. The relevant $K$ will be clear from the context. 

\begin{remark}
  \label{rem:be}
  For the sake of brevity we do not include Raviart-Thomas BE spaces on (uniformly
  shape regular) quadrilaterals and BDM-type BE spaces on triangles into our
  analysis. With slight alterations the approach of this paper covers these
  settings. Besides, curved elements can be treated by the usual mapping 
  techniques. 
\end{remark}


\section{Splitting technique}
\label{sec:abstract-theory}

Owing to the infinite-dimensional kernel of $\bDiv$ the bilinear form $a$ from
\eqref{eq:a} fails to be $\VX$-coercive, which massively compounds the difficulties
of convergence analysis for Galerkin schemes, as discussed, e.g., in
\cite[Sect.~3]{BUH03} and \cite{CHR00}. An abstract theory for tackling a priori
Galerkin error estimates for non-coercive variational problems like \eqref{eq:lvp} was
developed in \cite{BHP01,BUC01} and, in particular, in \cite[Sect.~3]{BUF03}. The
latter article tells us that Theorem~\ref{thm:main} will follow, once we establish
\begin{itemize}
\item[\textbf{(A)}] the existence of a stable direct splitting $\VX=\VV\oplus\VW$ such that
  {$a_{|V\times V}$} and $a_{|W\times W}$ are both $\VX$-coercive and
  $a_{|V\times W}$ and $a_{|W\times V}$ are both compact,
\item[\textbf{(B)}] the existence of a corresponding decomposition $\VX_{N} = \VV_{N} + \VW_{N}$,
  $\VW_{N}\subset\VW$, that is uniformly stable with respect to cell sizes and
  polynomial degree $p$,
\item[\textbf{(C)}] the \emph{gap property}
  \begin{gather}
    \label{eq:gap}
    \sup\limits_{\Vv_{N}\in\VV_{N}}  \inf\limits_{\Vv\in \VV}
    \frac{\N{\Vv-\Vv_{N}}_{X}}{\N{\Vv_{N}}_{X}} \leq C 
    \max_{K}\sqrt{\frac{h_{K}}{p_{K}+1}}.
  \end{gather}
\end{itemize}
We remark that the approximation property
\begin{gather}
  \label{eq:CAS}
  \inf\limits_{\Vv_{N}\in \VX_{N}} \N{\Vu-\Vv_{N}}_{X} \to 0\quad\text{as}\quad
  \max_{K}\sqrt{\frac{h_{K}}{p_{K}+1}} \to 0,
\end{gather}
dubbed CAS in \cite{BUF03,CFR00}, is automatically satisfied for families of
$hp$-Raviart-Thomas spaces on families of uniformly shape-regular meshes.

Taking the cue from the numerical analysis of electromagnetic wave equations
\cite[Sect.~5]{HIP02}, one might resort to the ``$L^{2}(\Gamma)$-orthogonal''
\emph{Hodge-decomposition} of $\VX$ \cite{BUC99a} in order to obtain a suitable
splitting. For $h$-version analysis this idea was successfully applied in
\cite{HIS01,BCS01,BHP01,BUC01}. Yet, on non-smooth surfaces, smoothness of functions
in the $\VV$-component may be poor, which causes substantial technical difficulties,
\textit{cf.}~\cite{BEH08}.
These are avoided when following the guideline that the analysis of boundary integral
operators is often greatly facilitated by taking a detour via a volume domain,
\textit{cf.} \cite{COS88a}. This strategy yields decompositions with enhanced
smoothness of the $\VV$-component.

More concretely, as in \cite{HIP01b} and \cite[Sect.~4.3.1]{BUF03}, $\VV$ and $\VW$
are constructed via a \emph{regularizing projection} $\PR:\VX\mapsto \VX$. To define
them, we intermittently visit volume domains abutting $\Gamma$. There the
construction employs $H^{1}$-regular vector potentials, see \cite[Sect.~2.4]{HIP02} and
\cite[Sect.~3]{ABD96}:

\begin{lemma}
  \label{lem:L}
  For any bounded Lipschitz domain $\Omega\subset\bbR^{3}$ there are continuous mappings
  $\LI:\curl \mathbf{H}(\curl,\Omega)\mapsto (H^{1}(\Omega))^{3}$ and $\LI_{0}:\curl
  \mathbf{H}_{0}(\curl,\Omega)\mapsto (H^{1}_{0}(\Omega))^{3}$ such that 
  $\curl \LI \Phibf =
  \Phibf$ for all $\Phibf\in \curl \mathbf{H}(\curl,\Omega)$ and 
  $\curl \LI_{0} \Phibf = \Phibf$ for all $\Phibf\in \curl \mathbf{H}_{0}(\curl,\Omega)$.
\end{lemma}

The construction is different for boundaries and screens and yields different
projection operators $\PR_{c}$ and $\PR_{o}$, respectively, fortunately sharing
the same pivotal properties. To begin with, fix $\Vu\in \VX$.
\begin{enumerate}
\item[(I)] Case of a closed surface $\Gamma=\partial\Omega$ \cite[Sect.~7]{HIP01b}:
  $\PR_{c}\Vu := {\bigl((\LI\Phibf)\times\Bn\bigr)|}_{\Gamma}$, where
  \begin{gather*}
    \Phibf := \grad w:\quad w\in H^{1}(\Omega):\quad
    \begin{array}[c]{rcll}
      -\Delta w &=& 0 & \text{in }\Omega\;,\\
      \grad w\cdot\Bn &=& \bDiv \Vu & \text{on }\Gamma\;.
    \end{array}
  \end{gather*}
  The fact that $\int\nolimits_{\Sigma}\bDiv \Vu\,\mathrm{d}S = 0$ for each connected
  component $\Sigma$ of $\Gamma$ guarantees $\Phibf\in
  \curl\mathbf{H}(\curl,\Omega)$, and Lemma~\ref{lem:L} can be applied.
\item[(II)] Case of a bounded open orientable Lipschitz surface $\Gamma$ 
  with boundary $\partial \Gamma$ and unit normal vector field $\Bn_{\Gamma}$:
  \begin{assumption}
    \label{ass:geo}
    There exist two bounded Lipschitz domains $\Omega_{1},\Omega_{2}\subset\bbR^{3}$
    satisfying
    \begin{itemize}
    \item $\overline{\Omega}_{1}\cap\overline{\Omega}_{2}=\overline{\Gamma}$,
    \item $\Omega:=\Omega_{1}\cup\Gamma\cup\Omega_{2}$ is a bounded Lipschitz domain
      with trivial topology,
    \item $\Gamma\subset\partial\Omega_{1}$ and $\Gamma\subset\partial\Omega_{2}$.
    \end{itemize}
  \end{assumption}
  In words, $\Gamma$ is the cut chopping the sphere-like $\Omega$ into two parts
  $\Omega_{1}$, $\Omega_{2}$, see Figure~\ref{fig:screen}.

  \begin{figure}[!htb]
    \centering
    \psfrag{G}{$\Gamma$}
    \psfrag{O1}{$\Omega_{1}$}
    \psfrag{O2}{$\Omega_{2}$}
    \psfrag{ng}{{$\Bn_{\Gamma}$}}
    \includegraphics[width=0.6\textwidth]{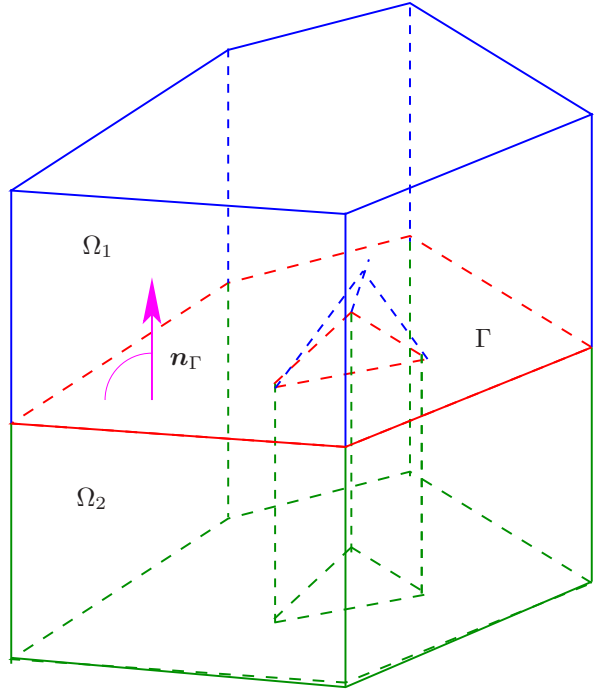}
    \caption{Screen $\Gamma$ with attached domains $\Omega_1$ and $\Omega_2$. Note
      the nontrivial topology of $\Gamma$ and how it can be dealt with in the
      construction of $\Omega$.
    }
    \label{fig:screen}
  \end{figure}

  The fact that $\int\nolimits_{\Gamma}\bDiv \Vu\,\mathrm{d}S = 0$ makes it
  possible to define for $i=1,2$
  \begin{align*}
    w_{i}\in H^{1}(\Omega_{i}):\quad 
     \begin{array}[c]{rcll}
      -\Delta w_{i} &=& 0 & \text{in }\Omega_{i}\;,\\
      \grad w_{i}\cdot\Bn & = & 0 & \text{on }\partial \Omega_{i}\setminus\Gamma\;,\\
      \grad w_{i}\cdot\Bn_{\Gamma} &=& \bDiv \Vu & \text{on }\Gamma,
    \end{array}
  \end{align*}
and then
  \[
    \Phibf := 
    \begin{cases}
      \grad w_{1} & \text{in }\Omega_{1}\\
      \grad w_{2} & \text{in }\Omega_{2}
    \end{cases}\quad  \in \mathbf{H}_{0}(\Div 0,\Omega)\;,
  \]
  because the normal component of $\Phibf$ is continuous across $\Gamma$. Hence, we
  can apply Lemma~\ref{lem:L} and set $\PR_{o}\Vu :=
  {\bigl((\LI_{0}\Phibf)\times\Bn_{\Gamma}\bigr)|}_{\Gamma}$.
\end{enumerate}

\medskip
By using elliptic lifting theorems, the continuity of $\LI$ and $\LI_{0}$,
and trace theorems we conclude, $\ast=c,o$:
\begin{gather}
  \label{eq:3}
  \exists C=C(\Gamma)>0:\quad
  \N{\PR_{\ast}\Vu}_{\mathbf{H}_{\perp}^{1/2}(\Gamma)} \leq C \N{\bDiv\Vu}_{H^{-1/2}(\Gamma)}
  \quad\forall \Vu\in \VX\;,
\end{gather}
where $\mathbf{H}_{\perp}^{1/2}(\Gamma)\subset \VX$ is the rotated tangential trace space 
\begin{itemize}
\item of $(H^{1}(\Omega))^{3}$ on $\Gamma:=\partial\Omega$ for closed surfaces
  \cite{BUC99,BCS00},
\item of $(H_{0}^{1}(\Omega))^{3}$ on the screen $\Gamma$, see \cite[Sect.~3.2]{BUC99}.
\end{itemize}
Moreover, by construction, on $\Gamma$
\begin{gather}
  \label{eq:2}
  \bDiv \PR_{\ast}\Vu = \bDiv\Vu\quad \forall \Vu\in \VX\quad
  \Rightarrow\quad \PR_{\ast}^{2} = \PR_{\ast}\;.
\end{gather}
Now we are in a position to define
\begin{gather}
  \label{eq:4}
  \VV := \PR_{\ast}(\VX) \subset \mathbf{H}_{\perp}^{1/2}(\Gamma),\quad
  \VW := (Id-\PR_{\ast})(\VX) 
  \overset{\text{by \eqref{eq:2}}}{=} \VX\cap 
  \mathbf{H}^{-1/2}(\bDiv0,\Gamma)\;.
\end{gather}
In light of \eqref{eq:3}, the continuous embedding
$\mathbf{H}_{\perp}^{1/2}(\Gamma)\hookrightarrow \mathbf{H}^{-1/2}(\bDiv,\Gamma)$ ensures
stability of the splitting.

Note that the embedding $\VV\hookrightarrow \mathbf{L}^{2}_{t}(\Gamma)$ is compact by
\eqref{eq:3} and Rellich's theorem. Thus, thanks to the
$(H^{1/2}(\Gamma))'$-coercivity (resp.,
$(\mathbf{H}_{\perp}^{1/2}(\Gamma))'$-coercivity) of the single layer boundary
integral operator $V_{k}$ (resp., $\VV_{k}$), see \cite[Prop.~2]{BCS01},
\cite[Lemma~8]{BUH03}, \cite[Lemma~7]{BUH03}, and \cite[Proof of Thm.~3.4]{BUC01}, we
infer the $\VX$-coercivity of {$a_{|V\times V}$} and $a_{|W\times W}$. Again,
appealing to the compact embedding $\VV\hookrightarrow
(\mathbf{H}_{\perp}^{1/2}(\Gamma))'$, the compactness of $a_{|V\times W}$ and
$a_{|W\times V}$ is immediate \cite[Lemma~9]{BUH03}. This yields \textbf{(A)}.

\begin{remark}
  \label{rem:goe}
  Assumption~\ref{ass:geo} is easily verified for piecewise smooth Lipschitz
  screens through extension in normal direction followed by patching holes
  by means of thick cutting surfaces in order to mend topological defects. 
  Yet, to keep the paper focused, we will not elaborate on this, but
  prefer to retain Assumption~\ref{ass:geo}.
\end{remark}

\begin{remark}
  \label{rem:traces}
  We recall from \cite{BUC99} that $\VX$ is the natural tangential trace space of
  $\Hcurl$ for a closed surface $\Gamma$, and of $\zbHcurl$ for a screen
  $\Gamma$.
\end{remark}


\newcommand{\PM}{\operatorname{\mathsf{P}}}

\section{Smoothed Poincar\'e lifting}
\label{sec:smooth-poinc-mapp}

For a domain $D\subset\bbR^{2}$ that is star-shaped with respect to $\Ba\in D$,
the Poincar\'e lifting
\begin{gather}
  \label{eq:P}
  (\PM_{\Ba}u)(\Bx) := \int\nolimits_{0}^{1}
  \tau u(\Ba+\tau(\Bx-\Ba))(\Bx-\Ba) 
  \,\mathrm{d}\tau
\end{gather}
provides a right inverse of the 2D divergence-operator $\Div \Vu := (\frac{\partial
  u_{1}}{\partial x_{1}}+\frac{\partial u_{2}}{\partial x_{2}})$ for continuous
functions: $\Div \PM_{\Ba} u = u$ for all $u\in C^{0}(\overline{D})$,
see \cite[Sect.~3]{DEB01}. In \cite{CMI08} M. Costabel and A. McIntosh demonstrated
how to mend the somewhat insufficient continuity properties of $\PM_{\Ba}$ by local
averaging:
\begin{assumption}
  \label{ass:ss}
  $D$ is star-shaped with respect to a ball $B\subset D$.
\end{assumption}

Then define the \emph{smoothed Poincar\'e lifting} \cite[Sect.~3]{CMI08} as
\begin{gather}
  \label{eq:PS}
  (\PM u)(\Bx) := \int\nolimits_{B} \psi(\Ba)
  (\PM_{\Ba} u)(\Bx)\,\mathrm{d}\Ba\;,
\end{gather}
where $\psi\in C^{\infty}(\bbR^{2})$, $\operatorname{supp}(\psi)\subset B$,
$\int_{B}\psi(\Bx)\,\mathrm{d}\Bx = 1$. We get the following powerful 
mapping properties from \cite[Cor.~3.3]{CMI08}.

\begin{theorem}
  \label{thm:PS}
  Under Assumption~\ref{ass:ss}, the smoothed Poincar\'e lifting $\PM$ according to
  \eqref{eq:PS} provides a continuous operator $\PM:H^{s}(D)\mapsto (H^{s+1}(D))^{2}$
  for any $s\in\bbR$ and satisfies $\Div \PM \varphi = \varphi$ for all $\varphi\in
  L^{2}(D)$.
\end{theorem}

A crucial property of the smoothed Poincar\'e mapping is the preservation of the
local boundary element spaces: the smoothed Poincar\'e mapping $\PM$ on the
(star-shaped) reference triangle $\wh{K}$ fulfills, \textit{cf.}
\cite[Sect.~3]{HIP08t}, \cite{HIP96b}, \cite[Sect.~4.2]{CMI08},
\begin{gather}
  \label{eq:8}
  \PM(\Div\RT{p}{\wh{K}}) \subset \RT{p}{\wh{K}}\quad
  \overset{\text{by \eqref{eq:17}}}{\Rightarrow}\quad \boxed{\PM(\Div \VX_{N}(\wh{K}))
    \subset \VX_{N}(\wh{K})}\;.
\end{gather}

\section{Projection based interpolation}
\label{sec:proj-based-interp}

Following \cite{HIP01b} and \cite[Sect.~4.3.1]{BUF03} again, local projection
operators will be used to build a suitable splitting of $\VX_{N}$. However,
$p$-refinement entails a more subtle approach that resorts to so-called commuting
projection based interpolation operators, see \cite{DEB04,DEB01,DEK07},
\cite[Sect.~3.6]{HIP02}, and \cite{DEM06} for a comprehensive exposition. Commuting
projectors link different finite element spaces on $\Cm$, the spaces $S_{N}$ and
$\widetilde{\VX}_{N}$ in the current setting.  Employing the relatively simple construction of
\cite{DEB01} will be sufficient for our purposes and the following results from that
article and from \cite{DEB04,BEH09b} will be used:
\begin{enumerate}
\item There are projection operators (with domains $\Cd(\cdot)$)
  \begin{align}
    \label{eq:15}
    \Pibf_{X}& :\Cd(\Pibf_{X})\subset H^{-1/2}_{\parallel}(\bDiv,\Gamma)
    \mapsto\widetilde{\VX}_{N}\;,\\
    \label{eq:19}
    \Pi_{S}&:\Cd(\Pi_{S})\subset H^{1}(\Gamma)\mapsto S_{N}\;,\\
    \label{eq:20}
    \Pi_{Q}&:L^{2}(\Gamma)\mapsto Q_{N}\;,
  \end{align}
  satisfying the \emph{commuting diagram properties} \cite[Prop.~3]{DEB01}
  \begin{gather}
    \label{eq:cdp}
    \bcurl \circ \Pi_{S} = \Pibf_{X}\circ\bcurl \quad\text{on}\quad 
    \Cd(\Pi_{S})\;,\\
    \label{eq:cdpQ}
    \bDiv\circ\Pibf_{X} = \Pi_{Q}\circ\bDiv\quad \quad\text{on}\quad 
    \Cd(\Pibf_{X})\;.
  \end{gather}
  For an open surface $\Gamma$ the interpolation operator complies with
  boundary conditions:
  \begin{gather}
    \label{eq:23}
    \Pibf_{X}(\VX\cap\Cd(\Pibf_{X})) = \VX_{N}\;.
  \end{gather}
\item As typical for the finite element interpolation operators,
  $\Pibf_{X}$ and $\Pi_{S}$ are strictly local in the sense that both these
  projectors can be obtained by patching together purely local cell based projectors
  $\Pibf_{K,X}$ and $\Pi_{K,S}$, $K\in\Cm$. This is because for any edge $E$
  of $\Cm$ with in-plane normal $\Bn_{E}$ the traces
  ${\Pibf_{X}\Vu\cdot\Bn_{E}|}_{E}$ and ${\Pi_{S}\varphi|}_{E}$ depend only on
  ${(\Vu\cdot\Bn_{E})|}_{\overline{E}}$ and ${\varphi|}_{\overline{E}}$,
  respectively.
  Furthermore, pullback commutes with local interpolation:
  \begin{gather}
    \label{eq:24}
    \Pibf_{\widehat{K},X}\circ \Phibf_{K}^{\ast} = \Phibf_{K}^{\ast}\circ
    \Pibf_{K,X}\quad\text{on}\quad \Cd(\Pibf_{K,X})\;.
  \end{gather}
\item The projectors $\Pi_{K,S}$ enjoy the approximation
  property (this follows from \cite[Thm.~4.1]{BEH09b} with a scaling argument)
  \begin{gather}
    \label{eq:6}
    \SN{\varphi-\Pi_{K,S}\varphi}_{H^{1}(K)} \leq C
    \sqrt{\frac{h_{K}}{p_{K}+1}} \SN{\varphi}_{H^{3/2}(K)}\quad
    \forall \varphi \in H^{3/2}(K)\;.
  \end{gather}
\end{enumerate}
These facts can be used to establish a special projection error estimate
for $\Pibf_{K,X}$, \textit{cf.} \cite[Sect.~5]{HIP08t}, \cite[Lemma~4.6]{HIP02},
\cite[Sect.~4]{ABD96}.
\begin{lemma}
  \label{lem:41}
  With $C>0$ depending only on the shape-regularity of the triangle $K\in\Cm$
  there holds
  \begin{gather*}
    \N{\Vu-\Pibf_{K,X}\Vu}_{L^{2}(K)} \leq C
    \sqrt{\frac{h_{K}}{p_{K}+1}}
    \N{\Vu}_{\mathbf{H}^{1/2}(K)}\;,
  \end{gather*}
  for all {$\Vu\in \mathbf{H}^{1/2}(K)$} with {$\boxed{\bDiv\Vu\in\bDiv\VX_{N}(K)}$}.
\end{lemma}

\begin{proof}
  Write $\PM$ for the smoothed Poincar\'e lifting (see
  Sect.~\ref{sec:smooth-poinc-mapp}) on $\wh{K}$. Fix $K\in\Cm$ and pick $\Vu\in
  \mathbf{H}^{1/2}(\widehat{K})$ with ${\Div\Vu\in\Div\VX_{N}(\wh{K})}$. This vector
  field is split according to
  \begin{gather}
    \label{eq:7}
    \Vu = \PM\Div\Vu + \underbrace{(\Vu- \PM\Div\Vu)}_{\hbox{$\Div$-free}}
      = \PM\Div\Vu + \curl_{2D} \varphi\;,
  \end{gather}
  where $\curl_{2D}$ denotes a rotated gradient and
  the existence of the scalar potential $\varphi\in \{\psi\in
  H^{1}(\wh{K}):\,\int_{\wh{K}}\psi\,\mathrm{d}\Vx=0\}$ is a consequence of
  $\Div(\Vu-\PM\Div\Vu)=0$, which follows from Theorem~\ref{thm:PS}.
  Theorem~\ref{thm:PS} also supplies the continuity of
  $\PM:H^{-1/2}(\wh{K})\mapsto \mathbf{H}^{1/2}(\wh{K})$, which paves the way for
  estimating
  \begin{align}
    \label{eq:11}
    \SN{\varphi}_{H^{3/2}(\wh{K})} & \leq 
    C \SN{\curl_{2D}\varphi}_{\mathbf{H}^{1/2}(\wh{K})} \leq
    C \left(\N{\Vu}_{\mathbf{H}^{1/2}(\wh{K})} + 
      \N{\PM\Div\Vu}_{\mathbf{H}^{1/2}(\wh{K})}\right)
    \nonumber
    \\[3pt]
    & \leq \N{\Vu}_{\mathbf{H}^{1/2}(\wh{K})} + 
    C \N{\Div\Vu}_{H^{-1/2}(\wh{K})} \leq 
    C \N{\Vu}_{\mathbf{H}^{1/2}(\wh{K})}\;,
  \end{align}
  where the first step is justified by interpolation between $H^{1}(\wh{K})$ and
  $H^{2}(\wh{K})$. Then, by the projector property of $\Pibf_{\wh{K},X}$,
  imbedding \eqref{eq:8}, and the discrete nature of $\Div\Vu$, there holds 
  \begin{align*}
    \Vu-\Pibf_{\wh{K},X}\Vu & = \underbrace{(Id-\Pibf_{\wh{K},X}) \PM\Div\Vu}_{=0}
    + (Id-\Pibf_{\wh{K},X})\curl_{2D} \varphi \\ 
    & \overset{\text{by \eqref{eq:cdp}}}{=} \curl_{2D}(Id-\Pi_{\wh{K},S})\varphi\;,
  \end{align*}
  where we owe the last identity to the commuting diagram property
  \eqref{eq:cdp} on $\wh{K}$. This makes it possible to apply \eqref{eq:6}
  \begin{gather*}
    \N{\Vu-\Pibf_{\wh{K},X}\Vu}_{L^{2}(\wh{K})}
    \begin{aligned}[t]
      & =  
    \SN{\varphi-\Pi_{\wh{K},S}\varphi}_{H^{1}(\wh{K})} \\ 
    & \leq  C\,(p_{K}+1)^{-1/2}\,
    \SN{\varphi}_{H^{3/2}(\wh{K})} \overset{\text{\eqref{eq:11}}}{\leq}
    C\,(p_{K}+1)^{-1/2}\,\SN{\Vu}_{\mathbf{H}^{1/2}(\wh{K})}\;.
  \end{aligned}
  \end{gather*}
  Here, switching to the semi-norm in $\mathbf{H}^{1/2}(\wh{K})$
  can be justified by a fractional Bramble-Hilbert lemma
  \cite[Prop.~6.1]{DUS80}. Eventually, \eqref{eq:24} and 
  a scaling argument take the estimate to the cell $K$.
\end{proof}

We have implicitly proved that the ($p$-dependent !) local projectors
$\Pibf_{K,X}:\{\Vu\in \mathbf{H}^{1/2}(K):\, \bDiv\Vu\in\bDiv\VX_{N}(K)\}\mapsto
\VX_{N}(K)$ are continuous with norm \emph{independent of $p$}.

\begin{remark}
  \label{rem:prj}
  In fact, the projector $\Pibf_{X}$ is closely linked to the splitting
  \eqref{eq:5}. From \cite{DEB01} and \cite[Sect.~4.]{DEM06} we extract the
  particular form
  \begin{gather*}
    \Pibf_{X} = \Pibf_{0} + \sum\limits_{E\in\Ce} \Pibf_{E}(Id-\Pibf_{0}) + 
    \sum\limits_{K} \Pibf_{K}(Id-\Pibf_{E})(Id-\Pibf_{0})\;,
  \end{gather*}
  where $\Pibf_{0}$, $\Pibf_{E}$, $\Pibf_{K}$ are suitable projection operators
  into $\RT{0}{\Cm}$, $\RT{p_{E}}{E}$, and $\RTn{p_{K}}{K}$, respectively.
\end{remark}


\section{Discrete splitting}
\label{sec:discrete-splitting}

Since $\bDiv \PR_{\ast} \VX_{N} = \bDiv\VX_{N}$, \eqref{lem:41} confirms that the
following definitions are valid for $\ast=c,o$: 
\begin{gather}
  \label{eq:9}
  \VV_{N} := \Pibf_{X}\PR_{\ast}(\VX_{N})\quad,\quad
  \VW_{N} := (Id - \Pibf_{X}\circ \PR_{\ast})\VX_{N}\;.
\end{gather}
By the commuting diagram property \eqref{eq:cdpQ} and \eqref{eq:2}, we find
\begin{gather}
  \bDiv\Pibf_{X}\PR_{\ast}\Vu_{N} = \Pi_{Q}\bDiv\PR_{\ast}\Vu_{N} =
  \Pi_{Q}\underbrace{\bDiv\Vu_{N}}_{\in Q_{N}} = \bDiv\Vu_{N}\quad
  \forall \Vu_{N}\in\VX_{N}\;.\notag\\
  \label{eq:21}
  \Rightarrow\quad \PR_{\ast}\Pibf_{X}\PR_{\ast} = \PR_{\ast}\quad\text{on
  }\VX_{N}\;.
\end{gather}
Hence, $\Pibf_{X}\PR_{\ast}:\VX_{N}\mapsto\VX_{N}$ is a projection, which confirms
that $\VX_{N} = \VV_{N} + \VW_{N}$. Stability, in the sense of
\begin{gather}
  \label{eq:hstab}
  \N{\Pibf_{X}\PR_{\ast}\Vu_{N}}_{X} \leq C \N{\Vu_{N}}_{X}\;,
\end{gather}
with $C>0$ depending on $\Gamma$ and the shape-regularity of $\Cm$ only, is another
consequence of Lemma \ref{lem:41} together with $\bDiv\Pibf_{X}\PR_{\ast}\Vu_{N} =
\bDiv\Vu_{N}$. This latter property also implies
$\VW_{N}\subset\VW=\VX\cap\mathbf{H}^{-1/2}(\bDiv0,\Gamma)$. This verifies assumption
\textbf{(B)} from Sect.~\ref{sec:abstract-theory}.  

It remains to establish \textbf{(C)}, the gap property \eqref{eq:gap}, which will be an 
immediate consequence of the following lemma.

\begin{lemma}
  \label{lem:gap}
  There is a constant $C>0$ depending only on the geometry of $\Gamma$
  and the shape-regularity of $\Cm$, such that for $\ast=c,o$
  \begin{gather*}
    \N{(Id-\Pibf_{X})\PR_{\ast}\Vu_{N}}_{X} 
    \leq C \max_{K}\sqrt{\frac{h_{K}}{p_{K}+1}}
    \N{\Vu_{N}}_{X}\quad
    \forall \Vu_{N}\in \VX_{N}\;.
  \end{gather*}
\end{lemma}

\begin{proof} By construction, we know that $\bDiv \PR_{\ast}\Vu_{N} = \bDiv\Vu_{N}$, which 
  permits us to apply the estimate of Lemma \ref{lem:41} to ${\PR_{\ast}\Vu_{N}}|_{K}$, $K\in\Cm$:
    \begin{gather*}
      \N{(Id-\Pibf_{K,X})\PR_{\ast}\Vu_{N}}_{L^{2}(K)}^{2} \leq C
      {\frac{h_{K}}{p_{K}+1}}
      \SN{\PR_{\ast}\Vu_{N}}_{\mathbf{H}^{1/2}(K)}^{2}\;.
    \end{gather*}%
  Patching together the local projectors and using sub-additivity of the
  $\SN{\cdot}_{\mathbf{H}_{\perp}^{1/2}}$-semi-norm, we arrive at 
  (we remind that $\PR_{\ast}\Vu_{N}\in\mathbf{H}_{\perp}^{1/2}(\Gamma)$)
   \begin{align*}
     \N{(Id-\Pibf_{X})\PR_{\ast}\Vu_{N}}_{L^{2}(\Gamma)} & \leq
     C \max_{K}\sqrt{\frac{h_{K}}{p_{K}+1}}
     \SN{\PR_{\ast}\Vu_{N}}_{\mathbf{H}_{\perp}^{1/2}(\Gamma)}  \\ & 
     \overset{\text{\eqref{eq:3}}}{\leq}
     C \max_{K}\sqrt{\frac{h_{K}}{p_{K}+1}}
     \N{\bDiv\Vu_{N}}_{H^{-1/2}(\Gamma)}\;.
   \end{align*}
  Since $\bDiv((Id-\Pibf_{X})\PR_{\ast}\Vu_{N}) = 0$, this is sufficient for the
  assertion of the lemma.

  We point out that when we deal with an open surface $\Gamma$, we recall that
  $\PR_{o}\Vu_{N}\in \VX$ is guaranteed by the construction of
  Sect.~\ref{sec:abstract-theory}. For a continuous tangential vector field $\Vu\in
  \VX$ that is smooth on the faces of $\Gamma$, the constraint in \eqref{eq:X2}
  implies vanishing in-plane normal components on $\partial \Gamma$. By locality of
  $\Pibf_{X}$, this will carry over to $\Pibf_{X}\Vu$, which means $\Pibf_{X}\Vu\in
  \VX_{N}$, \textit{cf.} \eqref{eq:23}. Further, we know from
  Sect.~\ref{sec:proj-based-interp}, that $\PR_{o}\Vu_{N}$ is in the domain of
  $\Pibf_{X}$. A simple density argument then confirms
  $\Pibf_{X}\PR_{o}\Vu_{N}\in\VX_{N}$, without adjusting the interpolation operator
  $\Pibf_{X}$, and the above proof carries over unaltered.
\end{proof}

The gap property \eqref{eq:gap} now immediately follows from the estimate
of Lemma~\ref{lem:gap}:
\begin{align*}
  \sup\limits_{\Vv_{N}\in\VV_{N}} \inf\limits_{\Vv\in\VV} 
  \frac{\N{\Vv-\Vv_{N}}_{X}}{\N{\Vv_{N}}_{X}} &
  \leq \sup\limits_{\Vv_{N}\in\VV_{N}}
  \frac{\N{\PR_{\ast}\Vv_{N}-\Vv_{N}}_{X}}{\N{\Vv_{N}}_{X}} \\ &
  \overset{\text{\eqref{eq:21}}}{=} 
  \sup\limits_{\Vv_{N}\in\VV_{N}}
  \frac{\N{\PR_{\ast}\Vv_{N}-\Pibf_{X}\PR_{\ast}
      \Vv_{N}}_{X}}{\N{\Vv_{N}}_{X}} \\ &
  \overset{\text{Lemma~\ref{lem:gap}}}{\leq}
  C \max_{K}\sqrt{\frac{h_{K}}{p_{K}+1}}.
\end{align*}


\bibliographystyle{siam}
\bibliography{references}

\end{document}